\documentclass[reqno,centertags, 12pt]{amsart}
\usepackage{amsmath,amsthm,amscd,amssymb}
\usepackage{latexsym}
\usepackage[mathscr]{eucal}
\newcommand{\bbC}{{\mathbb{C}}}
\newcommand{\bbD}{{\mathbb{D}}}

\newcommand{\bbG}{{\mathbb{G}}}

\newcommand{\bbR}{{\mathbb{R}}}

\newcommand{\bbZ}{{\mathbb{Z}}}

\newcommand{\calF}{{\mathcal F}}
\newcommand{\calH}{{\mathcal H}}
\newcommand{\calI}{{\mathcal I}}

\newcommand{\calK}{{\mathcal K}}
\newcommand{\calL}{{\mathcal L}}

\newcommand{\calS}{{\mathcal S}}

\newcommand{\bddot}{{\boldsymbol{\cdot}}}

\newcommand{\x}{{\mathbf{x}}}
\newcommand{\z}{{\mathbf{z}}}
\newcommand{\J}{{\mathscr{J}}}

\newcommand{\dott}{\,\cdot\,}

\newcommand{\lb}{\label}
\newcommand{\f}{\frac}

\newcommand{\ol}{\overline}
\newcommand{\ti}{\tilde  }
\newcommand{\wti}{\widetilde  }

\newcommand{\tr}{\text{\rm{Tr}}}

\newcommand{\dist}{\text{\rm{dist}}}
\newcommand{\idx}{\text{\rm{index}}}

\newcommand{\spec}{\text{\rm{spec}}}

\newcommand{\ran}{\text{\rm{Ran}}}

\newcommand{\e}{\text{\rm{e}}}
\newcommand{\di}{\text{\rm{d}}}
\newcommand{\s}{\text{\rm{s}}}

\newcommand{\bi}{\bibitem}

\newcommand{\beq}{\begin{equation}}
\newcommand{\eeq}{\end{equation}}
\newcommand{\ba}{\begin{align}}
\newcommand{\ea}{\end{align}}
\newcommand{\veps}{\varepsilon}

\newcommand{\fre}{\frak{e}}

\newcommand{\abs}[1]{\lvert#1\rvert}
\newcommand{\jap}[1]{\langle #1 \rangle}
\newcommand{\norm}[1]{\lVert#1\rVert}

\newcommand{\vy}{{\vec{y\!}}\,}

\newcounter{smalllist}
\newenvironment{SL}{\begin{list}{{\rm\roman{smalllist})}}{%
\setlength{\topsep}{0mm}\setlength{\parsep}{0mm}\setlength{\itemsep}{0mm}%
\setlength{\labelwidth}{2em}\setlength{\leftmargin}{2em}\usecounter{smalllist}%
}}{\end{list}}

\DeclareMathOperator{\Ima}{Im}

\allowdisplaybreaks
\numberwithin{equation}{section}

\newtheorem{theorem}{Theorem}[section]
\newtheorem{proposition}[theorem]{Proposition}
\newtheorem{lemma}[theorem]{Lemma}
\newtheorem{corollary}[theorem]{Corollary}
\theoremstyle{definition}

\theoremstyle{remark}
\newtheorem*{remark}{Remark}
\newtheorem*{remarks}{Remarks}

\newtheorem*{notes}{Notes}

\begin{document}

\title[Critical Lieb--Thirring Bounds in Gaps]
{Critical Lieb--Thirring Bounds in Gaps and the Generalized Nevai Conjecture for Finite Gap Jacobi Matrices}
\author[R.~Frank and B.~Simon]{Rupert L.\ Frank$^1$ and Barry Simon$^2$}

\thanks{$^1$ Department of Mathematics, Princeton University, Princeton, NJ 08544, USA.
E-mail: rlfrank@math.princeton.edu}
\thanks{$^2$ Mathematics 253-37, California Institute of Technology, Pasadena, CA 91125, USA.
E-mail: bsimon@caltech.edu. Supported in part by NSF grant DMS-0652919}

\date{July 4, 2010}

\begin{abstract} \ We prove bounds of the form 
\[\sum_{e\in I\cap\sigma_\di (H)} \dist (e,\sigma_\e (H))^{1/2}\leq L^1\text{-norm of a perturbation} 
\]
where $I$ is a gap. Included are gaps in continuum one-dimensional periodic
Schr\"odinger operators and finite gap Jacobi matrices where we get a generalized Nevai conjecture about an $L^1$
condition implying a Szeg\H{o} condition. One key is a general new form of the Birman--Schwinger bound in gaps.
\end{abstract}

\maketitle

\section{Introduction} \lb{s1}

This paper discusses spectral theory of Schr\"odinger operators, $-\Delta+V$ on $L^2(\bbR^\nu)$, and
Jacobi matrices
\begin{equation} \lb{1.1}
J=
\begin{pmatrix}
b_1 & a_1 & 0  & \cdots \\
a_1 & b_2 & a_2  & \cdots \\
0 & a_2 & b_3  & \cdots \\
\vdots & \vdots & \vdots  & \ddots
\end{pmatrix}
\end{equation}
on $\ell^2(\bbZ_+)$.

One of the streams motivating our work here are critical Lieb--Thirring inequalities. For any selfadjoint operator, $A$,
define
\begin{equation} \lb{1.2}
S^\gamma(A) = \sum_{e\in\sigma_\di(A)} \dist(e,\sigma_\e(A))^\gamma
\end{equation}
where $\sigma_\di$ is the discrete spectrum and $\sigma_\e$ the essential spectrum, and the sum counts any $e$ the number of
times of its multiplicity. Then, the original Lieb--Thirring bounds \cite{LT76} assert that (here $V_- =\max(0,-V)$)
\begin{equation} \lb{1.3}
S^\gamma (-\Delta+V) \leq L_{\gamma,\nu} \int V_-(x)^{\gamma+ \nu/2}\, d^\nu x
\end{equation}
for a universal constant, $L_{\gamma,\nu}$. In \cite{LT76}, Lieb and Thirring proved this for $\gamma >\f12$ if $\nu=1$
and for $\gamma >0$ if $\nu\geq 2$. The endpoint result for $\gamma=0$ if $\nu\geq 3$ is the celebrated CLR bound
(see \cite{Hun,LW} for reviews and history of Lieb--Thirring and related bounds). For $\nu=1$, the endpoint result (called
the critical bound) for $\gamma=\f12$ is due to Weidl \cite{Weidl}, with an alternate proof and optimal constant due to
Hundertmark, Lieb, and Thomas \cite{HLT}.

Here we will be interested in analogs of the critical bound in one dimension for perturbations of operators other than
$-\Delta$. For perturbations of the free Jacobi matrix ($J$ with $b_n\equiv 0$, $a_n\equiv 1$), the critical bound
is due to Hundertmark--Simon \cite{HS}, and for perturbations of periodic Jacobi matrices to Damanik, Killip, and Simon
\cite{DKS-Ann}. In \cite{FSW}, Frank, Simon, and Weidl proved bounds of the form
\begin{equation} \lb{1.5}
\sum_{\substack{e<\inf \sigma(H_0) \\ e\in\sigma(H)}} \dist(e,\sigma(H_0))^{1/2} \leq c\int \abs{V(x)}\, dx
\end{equation}
for $H_0 = -\f{d^2}{dx^2} + V_0$ and the Jacobi analog for $e<\inf \sigma (J_0)$ and $e>\sup \sigma(J_0)$, where
$H_0$ has a ``regular ground state'' and, in particular, in the case of periodic $V_0$.

Typical of our new results is:

\begin{theorem}\lb{T1.1} Let $V_0$ be a periodic, locally $L^1$ function on $\bbR$. Let $(a,b)$ be a gap in the
spectrum of $H_0=-\f{d^2}{dx^2}+V_0$. Then there is a constant $c$ so that for any $V\in L^1(\bbR)$, one has
\begin{equation} \lb{1.4}
\sum_{\substack{e\in\sigma_\di (H_0+V) \\ e\in (a,b)}} \dist(e,\sigma(H_0))^{1/2} \leq c \int \abs{V(x)}\, dx
\end{equation}
\end{theorem}

\begin{remark} This is an analog of a result of Damanik--Killip--Simon \cite{DKS-Ann} for perturbations of periodic
Jacobi matrices; they used what they call the magic formula to reduce to a critical Lieb--Thirring bound for
matrix perturbations of a free Jacobi matrix. They have a magic formula for periodic Schr\"odinger operators, but
it yields a nonlocal unperturbed object for which there is no obvious Lieb--Thirring bound.
\end{remark}

The other stream motivating this work goes back to a conjecture of Nevai \cite{Nev92} that if a Jacobi matrix, $J$, obeys
\begin{equation} \lb{1.6}
\sum_{n=1}^\infty \, \abs{a_n-1} + \abs{b_n} <\infty
\end{equation}
then its spectral measure,
\begin{equation} \lb{1.7}
d\rho(x) = f(x)\, dx + d\rho_\s (x)
\end{equation}
(with $d\rho_\s$ singular) obeys a Szeg\H{o} condition
\begin{equation} \lb{1.8}
\int_{-2}^2 (4-x^2)^{-1/2} \log(f(x))\, dx >-\infty
\end{equation}
This conjecture was proven by Killip--Simon \cite{KS}, that is,

\begin{theorem}[Killip--Simon \cite{KS}]\lb{T1.2} \eqref{1.6} implies \eqref{1.8}.
\end{theorem}

Their method, the model for analogs, is in two parts:
\begin{SL}
\item[(a)] Prove a theorem that
\begin{equation} \lb{1.9}
\prod_{n=1}^N a_n\to 1
\end{equation}
plus
\begin{equation} \lb{1.10}
\sum_{e\in\sigma_d(J)} \dist(e,\sigma_e(J))^{1/2} <\infty
\end{equation}
implies \eqref{1.8}. This generalizes results of Szeg\H{o}, Shohat, and Nevai (see \cite{Rice} for the history).

\item[(b)] Prove a critical Lieb--Thirring bound (in this case, done by Hundertmark--Simon \cite{HS}) to prove
\eqref{1.6} implies \eqref{1.10}.
\end{SL}

\smallskip
Since \eqref{1.6} clearly implies \eqref{1.9}, we get \eqref{1.8}. This strategy was exploited by Damanik--Killip--Simon
\cite{DKS-Ann} to prove an analog of Nevai's conjecture for perturbations of periodic Jacobi matrices. Here we are interested
in a larger class called finite gap Jacobi matrices. Let $\fre$ be a closed subset of $\bbR$ whose complement has $\ell$
open intervals plus two unbounded pieces: $\fre=\fre_1\cup\cdots\cup \fre_{\ell+1}$ and $\fre_j=[\alpha_j,\beta_j]$ with
$\alpha_1 < \beta_1 < \alpha_2 < \cdots < \alpha_{\ell+1} < \beta_{\ell+1}$. Periodic Jacobi matrices have $\sigma_\e
(J)$ equal to such an $\fre$, where each $\fre_j$ has rational harmonic measure, so such $\fre$'s are a small subset of all
finite gap $\fre$'s. In such a case, the set of periodic Jacobi matrices with $\sigma_\e(J)=\fre$ is a torus of
dimension $\ell$. For general $\fre$'s, there is still a natural  $\ell$-dimensional isospectral torus of almost
periodic $J$'s with $\sigma_\e (J)=\fre$. It is described, for example, in \cite{CSZ1}.

Here is another main result of this paper:

\begin{theorem}\lb{T1.3} Let $\{a_n^{(0)}, b_n^{(0)}\}_{n=1}^\infty$ be the Jacobi parameters for an element of
the isospectral torus of a finite gap set, $\fre$. Let $\{a_n,b_n\}$ be a set of Jacobi parameters obeying
\begin{equation} \lb{1.11}
\sum_{n=1}^\infty\, \abs{a_n-a_n^{(0)}} + \abs{b_n-b_n^{(0)}} <\infty
\end{equation}
Then the spectral measure, $d\rho$, of this perturbed Jacobi matrix has the form \eqref{1.7} where
\begin{equation} \lb{1.12}
\int_\fre \dist(x,\bbR\setminus\fre)^{-1/2} \log(f(x))\, dx > -\infty
\end{equation}
\end{theorem}

One part of our proof involves the general theory of eigenvalues in gaps, a subject with considerable literature
(see \cite{AADH,ADH,Bir90,Bir91,Bir91-ASM,Bir91-FAA,Bir95,Bir97,Bir98,BLS,BP,BR,BW,GGHK,GS,Hem89,Hem92,Hem97,HS08,
Klaus,Lev,Saf98,Saf01JMAA,Saf01,Sob1,Sob2}).
We will find a general Birman--Schwinger-type bound that could also be used to simplify many of these earlier works.
To describe this bound, we make several definitions.

If $C$ is selfadjoint and $I\subset\bbR$ and $I\cap\sigma_\e(C)=\emptyset$, we define
\begin{equation} \lb{1.13}
N(C\in I) = \dim(\ran(P_I(C)))
\end{equation}
with $P_I(\cdot)$ a spectral projection. $N(C>\alpha)=N(C\in (\alpha,\infty))$.

Recall that if $A$ is a selfadjoint operator bounded from below, a quadratic form $B$ is called relatively $A$-compact
if $Q(A)\subset Q(B)$, and for $e<\inf\sigma(A)$, $(A-e)^{-1/2} B(A-e)^{-1/2}$ is compact, that is, for some 
compact operator $K$ and all $u,v\in\calH$, 
\[
B((A-e)^{-1/2} u, (A-e)^{-1/2} v) = (u,Kv) 
\]
Often, $B$ is also an operator, in which case we may refer to an operator being form compact. The Birman--Schwinger 
principle says that if $B_- \geq 0$ is relatively $A$-compact and $E<\inf \sigma(A)$, then (see \cite{Hun})
\begin{equation} \lb{1.14}
N(A-B_- <E) = N(B_-^{1/2} (A-E)^{-1} B_-^{1/2} >1)
\end{equation}

There is a slight abuse of notation in \eqref{1.14} since a form need not have a square root. We need to suppose
our positive forms, $B$, can be written $C^*C$, where $C\colon\calH_{+1}\to\calK$ with $\calH_{+1},\calH_{-1}$ the usual
scale of spaces (see \cite{RS1}) and $\calK$ an arbitrary space (usually $\calK=\calH$). $B^{1/2}(A-E)^{-1} B^{1/2}$
is then $C(A-E)^{-1} C^*$. We call a form of this type ``factorizable'' when $C$ is compact as a map from
$\calH_{-1}$ to $\calK$. In our examples, since either $B$ is bounded and $C=\sqrt{B}$ or $B$ is multiplication by
$f\geq 0$ with $f\in L^1$ and $C=$ multiplication by $\sqrt{f}$, we'll use the simpler notation.

Suppose $E\notin\sigma(A)$ and $B\geq 0$ is relatively compact. As $x$ varies from $0$ to $1$, the discrete eigenvalues
of $A\pm xB$ are analytic in $x$ and strictly monotone, so there are only finitely many such $x$'s for which $E\in
\sigma (A\pm xB)$. We define $\delta_\pm (A,B;E)$ to be the number of solutions (counting multiplicity) with $x$
in $(0,1)$. \eqref{1.14} is proven by noting that
\begin{equation} \lb{1.15}
N(A-B_- <E) =\delta_- (A,B_-;E)
\end{equation}
and
\begin{equation} \lb{1.16}
\delta_-(A,B_-;E) = N(B_-^{1/2} (A-E)^{-1} B_-^{1/2} >1)
\end{equation}

Prior approaches to eigenvalues in gaps rely on going from $A$ to $A+B$ via $A\to A+B_+ \to A+ B_+ -B_-$ or
via $A\to A-B_-\to A + B_+ - B_-$. Thus, for example, by the same argument that leads to \eqref{1.15},
\begin{equation} \lb{1.17}
\begin{split}
N(A+B_+ - B_- \in  (\alpha,\beta)) &= \delta_+(A,B_+; \alpha) - \delta_+ (A,B_+;\beta) \\
&\quad + \delta_-(A+B_+, B_-;\beta) - \delta_- (A+B_+, B_-;\alpha)
\end{split}
\end{equation}

The analogs of \eqref{1.16} for $B\geq 0$ are
\begin{align}
\delta_- (A,B;E) &= N(B^{1/2} (A-E)^{-1} B^{1/2} >1 ) \lb{1.18} \\
\delta_+ (A,B;E) &= N(B^{1/2} (A-E)^{-1/2} B^{1/2} <-1) \lb{1.19}
\end{align}
Dropping the negative terms in \eqref{1.17} leads to
\begin{equation} \lb{1.20}
\begin{split}
N(A+B  \in (\alpha,\beta)) & \leq N(B_+^{1/2} (A-\alpha)^{-1} B_+^{1/2} <-1) \\
&\qquad + N(B_-^{1/2} (A+B_+ -\beta)^{-1} B_-^{1/2} >1)
\end{split}
\end{equation}
The $B_+ B_-$ cross-terms in \eqref{1.20} make it difficult to get Lieb--Thirring-type bounds although, with the
other results of this paper, one could prove Theorem~\ref{T1.3} from \eqref{1.20}. What allows us to get Lieb--Thirring
bounds is the following improvement of \eqref{1.20} that has no cross-terms:

\begin{theorem}\lb{T1.4} Let $B_+$ and $B_-$ be nonnegative, relatively form compact, factorizable perturbations of a
semibounded selfadjoint operator, $A$. Let $[\alpha,\beta]\subset\bbR\setminus\sigma(A)$. Suppose $\alpha,\beta\notin
\sigma(A+B_+)\cup\sigma(A-B_-)\cup\sigma(A+B_+-B_-)$. Then
\begin{equation} \lb{1.21}
\begin{split}
N(A+B_+ -B_- \in  (\alpha,\beta)) &\leq N(B_+^{1/2} (A-\alpha)^{-1} B_+^{1/2} < -1) \\
&\qquad + N(B_-^{1/2} (A-\beta)^{-1} B_-^{1/2} >1)
\end{split}
\end{equation}
\end{theorem}

\begin{notes} 1. $B_+,B_-$ need not be the positive and negative part of a single operator; in particular, they need
not commute.

\smallskip
2. While it is not stated as a formal theorem and not applied, Pushnitski \cite{Push} mentions \eqref{1.21}
explicitly (following Corollary~3.2 of his paper).
\end{notes}

We will prove this result in Section~\ref{s2}. We'll use this in Section~\ref{s3} to prove a CLR bound for perturbations
of $-\Delta+V_0$, where $V_0$ is a putatively generic periodic potential in $\bbR^\nu$, $\nu\geq 3$. Section~\ref{s4}
will provide an abstract result that shows that if there is an eigenfunction expansion near a gap, with eigenfunctions smooth
in a parameter $k$ with energies quadratic in $k$, then a critical Lieb--Thirring bound holds at that gap edge. The proof will
reduce to the original critical Lieb--Thirring bound, and so shed no light on why that bound holds (we regard both
proofs of that bound \cite{Weidl,HLT} as somewhat miraculous). In Section~\ref{s5}, we apply the abstract theorem to periodic
Schr\"odinger operators, and so get Theorem~\ref{T1.1}, and in Section~\ref{s6}, to finite gap Jacob matrices, and so
get Theorem~\ref{T1.3}. Section~\ref{s7} applies the decoupling results of Section~\ref{s2} to Dirac operators.

\medskip
We thank Alexander Pushnitski and Robert Seiringer for valuable discussions.

\section{Two Decoupling Lemmas} \lb{s2}

We'll need two basic decoupling facts: one, basically well known, and the second, Theorem~\ref{T1.4}.
All our operators act on a separable Hilbert space. The following is essentially a variant of the argument
used to prove the Ky Fan inequalities and is stated formally for ease of later use. It is well known.

\begin{proposition}\lb{P2.1} If $C$ and $D$ are compact selfadjoint operators and $c,d$ are in $(0,\infty)$, then
\begin{equation} \lb{2.1}
N(C+D > c+d )\leq N(C>c) + N(D>d)
\end{equation}
\end{proposition}

\begin{proof} Let $m=N(C>c)$, $n=N(D>d)$, and $\varphi_1, \dots, \varphi_m$ (resp.\ $\psi_1, \dots, \psi_n$), a basis
for $\ran(P_{(c,\infty)}(C))$ (resp.\ $\ran(P_{(d,\infty)}(D))$). If $\eta\perp \{\varphi_j\}_{j=1}^m \cup
\{\psi_j\}_{j=1}^n$, then $\jap{\eta, C\eta}\leq c$ and $\jap{\eta, D\eta}\leq d$. It follows from the min-max
principle that $C+D$ has at most $n+m$ eigenvalues above $c+d$.
\end{proof}

\begin{corollary}\lb{C2.2} If $S,T$ are compact operators and $c,d >0$, then
\begin{equation} \lb{2.2}
N((S+T)^* (S+T) > c+d) \leq N(S^*\! S > \tfrac12\, c) + N(T^* T>\tfrac12\, d)
\end{equation}
\end{corollary}

\begin{proof} Immediate from \eqref{2.1} and
\begin{equation} \lb{2.3}
(S+T)^* (S+T) \leq (S+T)^* (S+T) + (S-T)^* (S-T) = 2(S^*\!S + T^*T)
\end{equation}
\end{proof}

The key to our proof of Theorem~\ref{T1.4} (which we recall appears in \cite{Push}) is the following
Proposition \ref{P2.3}, for which we give a
proof
involving finite approximation at the end of this section. The appendix has an alternate proof that is more natural
to those who know about the relative index of projections \cite{ASS}, but it involves some machinery that is not so
commonly known. $\delta_\pm$ are defined just before \eqref{1.15}.

\begin{proposition}\lb{P2.3} Let $A$ be a semibounded selfadjoint operator and $B_\pm$ two nonnegative relatively
$A$-compact factorizable forms. Let $E\notin\sigma(A), \sigma(A+B_+), \sigma(A-B_-), \sigma(A+B_+-B_-)$. Then
\begin{equation} \lb{2.4}
\delta_+(A,B_+;E)-\delta_-(A+B_+, B_-; E) = -\delta_-(A,B_-;E) + \delta_+(A-B_-,B_+;E)
\end{equation}
\end{proposition}

\begin{remark} This asserts the intuitive fact that the net number of eigenvalues crossing $E$ in going from $A$ to
$A+B_+-B_-$ does not depend on the order in which we turn on $B_+$ and $B_-$. It is obvious in the finite-dimensional
case and we'll prove it by approximation by finite-dimensional matrices. It allows us to use different orders $A\to
A+B_+\to A+B_+-B_-$ and $A\to A-B_-\to A+B_+-B_-$ at $\alpha$ and at $\beta$.
\end{remark}

\begin{proof}[Proof of Theorem~\ref{T1.4}] By \eqref{2.4} (with $E=\beta)$ and \eqref{1.17},
\begin{equation} \lb{2.5}
\begin{split}
N(A+B_+ -B_-\in (\alpha,\beta)) &= \delta_+(A,B_+;\alpha) - \delta_-(A+B_+, B_-;\alpha) \\
&\qquad + \delta_- (A,B_-;\beta) -\delta_+ (A-B_-, B_+;\beta)
\end{split}
\end{equation}
\eqref{1.21} then follows from \eqref{1.18} and \eqref{1.19} and dropping two negative terms.
\end{proof}

We now turn to the proof of Proposition~\ref{P2.3}.

\begin{lemma}\lb{L2.4} Let $A$ be semibounded and selfadjoint, $B$ a relatively $A$-compact, positive, factorizable
quadratic form, and $E\notin\sigma(A),\sigma(A\pm B)$. Then there exist $B_n$, positive, finite rank bounded operators,
so that $\delta_\pm (A,B_n;E)=\delta_\pm (A,B;E)$ and $B_n^{1/2} (A-E)^{-1} B_n^{1/2}$ converge in norm to $B^{1/2} (A-E)^{-1}
B^{1/2}$.
\end{lemma}

\begin{proof} By \eqref{1.18} and \eqref{1.19}, it suffices to prove the norm convergence. Let $\calH_{\pm 1}$ be the
scale associated to $A$ (see \cite{RS1}). $B\colon\calH_{-1}\to\calH_{+1}$ with $B=C^*C$. $C$ is compact, so it can be
approximated by finite rank operators with vectors in $\calH$ and $\calK$.
\end{proof}

\begin{lemma} \lb{L2.5} Let $A$ be a semibounded operator with $E\notin\sigma(A)$ and $F\subset\calH$ a finite-dimensional
space. Then there exist $A_n$, finite rank operators, with $F\subset\ran(A_n-EQ_n)$ {\rm{(}}where $Q_n$ is the projection
onto $\ran(A_n)${\rm{),}} so that $B^{1/2} (A_n-EQ_n)^{-1} B^{1/2} \to B^{1/2} (A-E)^{-1} B^{1/2}$ in norm as $n\to\infty$
for all finite rank, nonnegative $B$ with $\ran(B)\subset F$.
\end{lemma}

\begin{proof} Define $f_n(x)\colon\bbR\to\bbR$ by
\[
f_n(x) = \begin{cases}
-n & \text{if } x\leq -n \\
n & \text{if } x\geq n \\
\f{1}{n}\, [nx] & \text{if } -n\leq x\leq n
\end{cases}
\]
where $[y]=$ integral part of $y$. Let $\ti A_n=f_n(A)$ so $\norm{(\ti A_n-E)^{-1} -(A-E)^{-1}}\to 0$. Let $Q_n$ be the
projection onto the cyclic subspace generated by $\ti A_n$ and $F$. This cyclic subspace is finite-dimensional, so $A_n =
Q_n\ti A_nQ_n$ is finite rank, and if $\ran(B)\subset F$, $B^{1/2} (\ti A_n-E)^{-1} B^{1/2}=B^{1/2}(A_n-EQ_n)^{-1} B^{1/2}$.
\end{proof}

\begin{proof}[Proof of Proposition~\ref{P2.3}] If $A$, $B_+$, and $B_-$ are operators on a finite-dimensional space, then
\eqref{2.4} is immediate, since both sides equal $\dim[\ran(P_{(-\infty,E)}(A))]-\dim[\ran(P_{(-\infty,E)}(A+B_+ -B_-))]$.
By the last two lemmas, we can find finite-dimensional $A_n$ and $(B_n)_\pm$ so that all $\delta$ objects in \eqref{2.4}
equal the $A,B_\pm$ objects.
\end{proof}

\section{CLR Bounds for Regular Gaps\\in Periodic
Schr\"odinger Operators} \lb{s3}

Let $V_0$ be a periodic, locally $L^{\nu/2}$ function on $\bbR^\nu$ for $\nu\geq 3$, that is,
\begin{equation} \lb{3.1}
V_0 (x+\tau_j) =V_0(x)
\end{equation}
for $\tau_1, \dots, \tau_\nu$ linearly independent in $\bbR^\nu$. Let $H_0=-\Delta+V_0$. Then $H_0$ is a direct integral
of operators, $H_0(k)$, with compact resolvent where $k$ runs through a fundamental cell of the dual lattice (see, e.g.,
\cite{RS4}). Let $\veps_1(k) \leq \veps_2(k) \leq \dots$ be the eigenvalues of $H_0(k)$. Let $(\alpha,\beta)$ be a gap in
$\sigma(H_0)$ in that $(\alpha,\beta)\cap\sigma (H_0)=\emptyset$ but $\alpha,\beta\in\sigma(H_0)$. We say $\beta$
(resp.\ $\alpha)$ is a regular band edge if and only if
\begin{SL}
\item[(i)] $\beta=\inf_k \veps_n(k)$ (resp.\ $\alpha=\sup_k \veps_n(k)$) for a single $n$.
\item[(ii)] $\veps_n(k)=\beta$ (resp.\ $\veps_n(k)=\alpha$) has finitely many solutions $k^{(1)}, \dots, k^{(\ell)}$.
\item[(iii)] At each $k^{(j)}$, $\veps_n(k)$ has a matrix of second derivatives which is strictly positive (resp.\ strictly
negative).
\end{SL}
We say that $(\alpha,\beta)$ is a regular gap if both band edges are regular. It is believed that for a generic $V_0$, all
band edges are regular (for generic results on (i), (ii), see Klopp--Ralston \cite{KR}). Birman \cite{Bir95} has proved
that if $(\alpha,\beta)$ is a regular gap, then with $\norm{\cdot}_{\calI_{\nu/2}^w}$ the weak trace class norm
(see \cite{STI}), one has a constant $c$ so that
\begin{equation} \lb{3.2}
\sup_{\lambda\in (\alpha,\beta)} \bigl\| \abs{W}^{1/2} (H_0-\lambda)^{-1} \abs{W}^{1/2}\bigr\|_{\calI_{\nu/2}^w}
\leq c\norm{W}_{\nu/2}
\end{equation}
By combining this with Theorem~\ref{T1.4}, one immediately has

\begin{theorem}\lb{T3.1} If $(\alpha,\beta)$ is a regular gap of $H_0$, then for any $W\in L^{\nu/2}(\bbR^\nu)$, we have
\begin{equation} \lb{3.3}
N(H_0+W\in (\alpha,\beta)) \leq c \int_{\bbR^\nu} \abs{W(x)}^{\nu/2}\, d^\nu x
\end{equation}
\end{theorem}

Because he didn't have Theorem~\ref{T1.4}, Birman restricted himself to perturbations of a definite sign.

Obviously, if there are finitely many gaps, one can sum over all gaps if they were all regular. It is known (see
Sobolev \cite{Sob3} and references therein) that if $V_0$ is smooth, then there are always only finitely many gaps.

\section{An Abstract Critical Lieb--Thirring Bound} \lb{s4}

In this section, we'll prove the following continuum critical Lieb--Thirring bound and discrete analog:

\begin{theorem}\lb{T4.1} Let $H_0$ be a semibounded selfadjoint operator on $L^2(\bbR,dx)$ so that for some $a<b$,
\begin{SL}
\item[{\rm{(i)}}]
\begin{equation} \lb{4.1}
[a,b) \cap \sigma(H_0)=\emptyset
\end{equation}

\item[{\rm{(ii)}}] For $E_0 <\inf\sigma(H_0)$, $(H_0-E_0)^{-1/2}$ is a bounded operator from $L^2$ to $L^\infty$.

\item[{\rm{(iii)}}] There exist $\veps,\delta >0$ and continuous functions $\rho,\theta,E$ from $(-\delta,\delta)$ to $\bbR$
and $u(\dott,\dott)$ from $\bbR\times(-\delta,\delta)$ to $\bbC$ so that any $\varphi\in \ran(P_{[b,b+\veps)}(H_0))$
has an expansion
\begin{equation} \lb{4.2}
\varphi(x) = \int_{-\delta}^\delta \wti\varphi (k) u(x,k)\, dk
\end{equation}
with
\begin{equation} \lb{4.3}
\wti{H_0\varphi}(k) =E(k)\, \wti\varphi(k)
\end{equation}
and
\begin{equation} \lb{4.4}
\norm{\varphi}_{L^2(\bbR,dx)}^2 = \int \abs{\wti\varphi(k)}^2 \rho(k)\, dk
\end{equation}
Moreover, for any $\wti\varphi\in L^2(-\delta,\delta;dk)$, \eqref{4.2} defines a function in $L^2(\bbR)$ lying in
$\ran(P_{[b,b+\veps)}(H_0))$ {\rm{(}}the integral converges by the hypothesis \eqref{4.7} below{\rm{)}}.

\item[{\rm{(iv)}}]
\begin{equation} \lb{4.5}
0<\inf_{k\in (-\delta,\delta)} \rho(k) = \rho_- <\sup_{k\in (-\delta,\delta)} \rho(k) = \rho_+ <\infty
\end{equation}

\item[{\rm{(v)}}] $E(k)=E(-k)$ and maps $[0,\delta)$ bijectively onto $[0,\veps)$. For some $c_1>0$, we have
\begin{equation} \lb{4.6}
E(k)\geq b+ c_1 k^2
\end{equation}

\item[{\rm{(vi)}}]
\begin{equation} \lb{4.7}
\sup_{\substack{k\in (-\delta,\delta) \\ x\in\bbR}} \abs{u(x,k)} = c_2 <\infty
\end{equation}

\item[{\rm{(vii)}}] If
\begin{equation} \lb{4.8}
v(x,k) = e^{-i\theta(k)x} u(x,k)
\end{equation}
then for some $c_3<\infty$ and all $x\in\bbR$,
\begin{equation} \lb{4.9}
\abs{v(x,k)-v(x,0)} \leq c_3 k^2
\end{equation}

\item[{\rm{(viii)}}] $\theta$ is $C^2$ on $(-\delta,\delta)$ and
\begin{equation} \lb{4.10}
\inf_{k\in(-\delta,\delta)} \theta'(k) >0
\end{equation}

\item[{\rm{(ix)}}]
\begin{equation} \lb{4.11a}
E(-k) = E(k), \quad u(x,-k) = \ol{u(x,k)}, \quad \theta(-k) = -\theta(k), \quad \rho(-k)=\rho(k)
\end{equation}
\end{SL}

Then for some $C$ and all $V\in L^1 (\bbR,dx)$, we have
\begin{equation} \lb{4.11}
\sum_{\substack{e\in\sigma_\di(H_0+V) \\ e\in (a,b)}} (b-e)^{1/2} \leq C \int \abs{V(x)}\ dx
\end{equation}
\end{theorem}

\begin{remarks} 1. There is a similar result for $(b,a]\cap\sigma(H_0)=\emptyset$ with \eqref{4.6} replaced by
\begin{equation} \lb{4.12}
E(k)\leq b-c_1 k^2
\end{equation}
This means we can control full gaps $(b_-,b_+)$ in $\sigma(H_0)$. To control $(-\infty,\inf\sigma(H_0))$ (and the top
half in the discrete case) will require an additional argument that we provide at the end of this section.

\smallskip
2. We could replace $\theta(k)$ by $k$ (and we'll essentially do that). We haven't because, in the finite gap case,
there is a natural parameter distinct from $\theta$.

\smallskip
3. The idea behind the proof will be to use decoupling to reduce the proof to control of the $[b,b+\veps)$ region
and use the eigenfunction expansion there to compare to $-\f{d^2}{dx^2} + \wti V(x)$, where $\wti V$ and $V$ have
comparable $L^1$ norms.

\smallskip
4. Hypothesis (ii) implies that any $V\in L^1$ is a relatively compact perturbation of $H_0$.

\smallskip
5. The decomposition we use in the proof below was suggested to us by a paper of Sobolev \cite{Sob1}, who used it
in a related, albeit distinct, context.

\smallskip
6. \eqref{4.2} and \eqref{4.3} imply for all $\wti\varphi \in L^2 ((-\delta,\delta),dk)$ and all $\psi\in
\ran(P_{[b,b+\veps)}(H_0))$, we have that
\begin{align*}
\jap{\psi,\varphi} &= \int_{-\delta}^\delta dk \int dx \, \wti\varphi(x)\, \ol{\psi(x)}\, u(x,k) \\
&= \int \ol{\wti{\psi}(k)}\, \wti\varphi(k)\rho(k)\, dk
\end{align*}
which implies that
\begin{equation} \lb{4.12a}
\wti\psi(k) = \rho(k)^{-1} \int dx\, \ol{u(x,k)}\, \psi(x)
\end{equation}
\end{remarks}

We'll prove \eqref{4.11} by reducing it to a bound on $N(H_0+V\in [a,b-\tau])$:

\begin{lemma} \lb{L4.2} If we  have $C_1,C_2,C_3$ so that for $0<\tau < b-a$,
\begin{equation} \lb{4.13}
N(H_0 +V\in [a,b-\tau]) \leq C_1 \int \abs{V(x)}\, dx
+ N\biggl( -\f{d^2}{dx^2} -C_2 V_- \leq -\f{\tau}{C_3}\biggr)
\end{equation}
then \eqref{4.11} holds.
\end{lemma}

\begin{remark} For control of a lower band edge, $V_-$ in the last term will be replaced by $V_+$.
\end{remark}

\begin{proof} For any absolutely continuous function, $f$, on $[a,b]$ with $f(b)=0$,
\begin{equation} \lb{4.14}
\sum_{\substack{e\in\sigma_\di (H_0+V) \\ e\in [a,b]}} f(e) = -\int_0^{b-a} f'(b-\tau)
N(H_0+V\in [a,b-\tau])\, d\tau
\end{equation}
so, by \eqref{4.13} with $f(y)=(b-y)^{1/2}$,
\begin{align*}
& \text{LHS of \eqref{4.11}} \\
& \quad \leq \int_0^{b-a} \tfrac12\, \tau^{-1/2} \biggl[ C_1 \norm{V}_1 + N\bigg( -\f{d^2}{dx^2}
- C_2 V_- \leq  -\f{\tau}{C_3}\biggr)\biggr] \, d\tau \\
& \quad = \bigl(\sqrt{b-a}\bigr) C_1 \norm{V}_1 + \sqrt{C_3} \int_0^{(b-a)/C_3} \tfrac12\, \sigma^{-1/2}
N\biggl( -\f{d^2}{dx^2} - C_2 V_- \leq -\sigma\biggr) d\sigma \\
& \quad \leq \bigl( \sqrt{b-a}\bigr) C_1 \norm{V_1}+ \sqrt{C_3} \sum_{\substack{e < 0 \\
e\in\sigma (-\f{d^2}{dx^2} - C_2 V_-)}} (-e)^{1/2} \\
& \quad \leq \bigl( \sqrt{b-a}\, C_1 + C_2 \sqrt{C_3}\, L_{\f12,1}\bigr) \norm{V}_1
\end{align*}
proving \eqref{4.11}. (It is known that $L_{\f12,1}=\f12$ \cite{HLT}.)
\end{proof}

\begin{lemma}\lb{L4.3} Suppose $E_0 <\inf\sigma (H_0)$ and $(H_0-E_0)^{-1/2}$ is a bounded operator from $L^2$ to
$L^\infty$. Let $f(x)$ be a function on $\sigma(H_0)$ with
\begin{equation} \lb{4.15}
D=\sup_{y\in\sigma(H_0)} \abs{f(y)}(y-E_0)<\infty
\end{equation}
Then for any $V\in L^1$,  $\abs{V}^{1/2} f(H_0) \abs{V}^{1/2}$ is trace class and
\begin{equation} \lb{4.16}
\norm{\abs{V}^{1/2} f(H_0) \abs{V}^{1/2}}_1 \leq D \norm{(H-E_0)^{-1/2}}_{2,\infty}^2 \norm{V}_1
\end{equation}
{\rm{(}}where the $\norm{\cdot}_1$ on the left is trace class norm and on the right is $L^1(\bbR)$ norm{\rm{)}}.
\end{lemma}

\begin{proof} By the Dunford--Pettis theorem (\cite{Tre}), $(H_0-E_0)^{-1/2}$ has a Hermitian symmetric integral kernel
$K(x,y)$ with
\[
\sup_x \biggl( \int \abs{K(x,y)}^2\, dy\biggr)^{1/2} = \norm{(H-E_0)^{-1/2}}_{2,\infty}
\]
so, by the symmetry, $(H-E_0)^{-1/2} \abs{V}^{1/2}$ is Hilbert--Schmidt with Hilbert--Schmidt norm bounded by
$\norm{(H-E_0)^{-1/2}}_{2,\infty} \norm{V}_1^{1/2}$. Since $D$ is the operator norm of $(H_0-E_0)f(H_0)$, \eqref{4.16}
is immediate.
\end{proof}

\begin{proof}[Proof of Theorem~\ref{T4.1}] We use \eqref{1.21} with $A=H_0$, $B=V$\!, $\alpha=a$, $\beta=b-\tau$, 
where $\tau$ is any point in $(0,b-a)$, and Lemma~\ref{L4.3} to see
\begin{equation} \lb{4.17}
\text{LHS of \eqref{4.13}} \leq N(V_-^{1/2} (H_0-b+\tau)^{-1} V_-^{1/2} >1) + C\int \abs{V_+(x)}\, dx
\end{equation}
for a suitable constant.

In the first term of \eqref{4.17}, we insert $P_{[b,b+\veps]}(H_0) + (1-P_{[b,b+\veps]}(H_0))$ in $(H_0-b+\tau)^{-1}$, use
\eqref{2.1} with $c=d=\f12$ and use Lemma~\ref{L4.3} to get
\begin{equation} \lb{4.18}
\begin{split}
N(V_-^{1/2} & (H_0-b+\tau)^{-1} V_-^{1/2} >1) \\
&\leq C \int \abs{V_-(x)}\, dx
+ N(V_-^{1/2} (H_0-b+\tau)^{-1} P_{[b,b+\veps]} (H_0) V_-^{1/2} > \tfrac12)
\end{split}
\end{equation}

By \eqref{4.2}--\eqref{4.4} and \eqref{4.12a}, for $\lambda\equiv b-\tau\notin\sigma(H_0)$, $(H_0-\lambda)^{-1}
P_{[b,b+\veps)} (H_0)$ has the integral kernel
\begin{equation} \lb{4.19}
\int_{-\delta}^\delta \f{u(x,k)\, \ol{u(y,k)}}{E(k)-b+\tau}\, \f{dk}{\rho(k)}
\end{equation}
Write
\begin{equation} \lb{4.20}
u(x,k) = e^{i\theta(k)x} v(x,0) + e^{i\theta (k)x}[v(x,k) - v(x,0)]
\end{equation}
and insert into \eqref{4.19}, writing the kernel as $(S_\tau + T_\tau)^* (S_\tau + T_\tau)$ and use \eqref{2.2}, where
$S,T$ have integral kernels
\begin{equation} \lb{4.21}
S_\tau (k,x) = (E(k)-b+\tau)^{-1/2} \rho(k)^{-1/2} e^{i\theta(k)x} v(x,0)
\end{equation}
and similarly for $T$.

By \eqref{4.5}, \eqref{4.6}, and \eqref{4.9}, uniformly in $k,x$ and $\tau$, $\abs{T_\tau(k,x)}$ is bounded, so
$T_\tau V_-^{1/2}$ is bounded uniformly in $\tau$ in Hilbert--Schmidt norm as a map from $L^2 (\bbR,dx)$ to $L^2
([b,b+\veps),dk)$. Thus, uniformly in $\tau$,
\begin{equation} \lb{4.22}
N(V_-^{1/2}T_\tau^* T_\tau V_-^{1/2}> \tfrac18) \leq C\int \abs{V_-(x)}\, dx
\end{equation}

Let $Q(\theta)$ be an inverse function to $\theta$. Changing variables from $k$ to $\theta$, $S_\tau^* S_\tau$ has
integral kernel
\begin{equation} \lb{4.23}
\int_{-\theta(\delta)}^{\theta(\delta)} \f{v(x,0)\, \ol{v(y,0)}\, e^{i\theta(x-y)}}{E(Q(\theta))-b+\tau}\,
\f{d\theta}{\theta'(Q(\theta)) \rho(Q(\theta))}
\end{equation}
By \eqref{4.10} and \eqref{4.6}, there is a constant $c_4$ with $E(Q(\theta)) -b+\tau \geq c_4 \theta^2 + \tau$. Also,
$u\bar u$ is a positive definite kernel, so the operator in \eqref{4.23} is dominated in operator sense by the kernel
\begin{equation} \lb{4.24}
c_5 \int_{-\infty}^\infty \f{v(x,0)\, \ol{v(y,0)}\, e^{i\theta(x-y)}}{c_4 \theta^2 + \tau} \, d\theta
\end{equation}
which is the integral kernel of $c_5 v(\dott,0) (-c_4 \f{d^2}{dx^2} + \tau)^{-1}\, \ol{v(\dott,0)}$. Thus,
\begin{align}
N(V_-^{1/2} & S_\tau^* S_\tau  V_-^{1/2} > \tfrac18) \notag \\
& = N\biggl(8c_5 v(\dott,0) V_-^{1/2} \biggl( -c_4\, \f{d^2}{dx^2} + \tau\biggr)^{-1} v(\dott,0)
V_-^{1/2} > 1\biggr) \notag \\
&= N\biggl( -\f{d^2}{dx^2} -\f{8c_5}{c_4}\, \abs{v(\dott, 0)}^2 V_- < - \f{\tau}{c_3}\biggr) \lb{4.25}
\end{align}
by the Birman--Schwinger principle.

Letting $C_2 = \f{8c_5}{c_4} \sup_x \abs{v(\dott,0)}^2$, we see that \eqref{4.13}, and so \eqref{4.11}, holds.
\end{proof}

Next, we turn to the analog for Jacobi matrices. $J_0$ is a fixed two-sided Jacobi matrix and $\delta J_0$ a
Jacobi perturbation with parameters $\{a_n^{(0)}, b_n^{(0)}\}_{n=-\infty}^\infty$ and $\{\delta a_n,
\delta b_n\}_{n=-\infty}^\infty$, respectively. $J=J_0+\delta J$ with parameters $\{a_n, b_n\}_{n=-\infty}^\infty$.

\begin{theorem} \lb{T4.4} Let $J_0$ be a Jacobi matrix on $\ell^2(\bbZ)$ so that for some $a<b$:
\begin{SL}
\item[{\rm{(i)}}]
\begin{equation} \lb{4.29}
[a,b) \cap\sigma(J_0)=\emptyset
\end{equation}

\item[{\rm{(ii)}}] There exist $\veps,\delta >0$ and functions $\rho,\theta,E$ from $(-\delta,\delta)$ to $\bbR$ and
$u_\bddot(\cdot)$ from $\bbZ\times (-\delta,\delta)$ to $\bbC$ so that any $\varphi\in\ran(P_{[b,b+\veps)}(J_0))$ has
an expansion
\begin{equation} \lb{4.30}
\varphi_n = \int_{-\delta}^\delta \wti\varphi(k) u_n(k)\, dk
\end{equation}
with
\begin{equation} \lb{4.31}
\wti{J_0 \varphi}(k) = E(k) \wti\varphi(k)
\end{equation}
and
\begin{equation} \lb{4.32}
\norm{\varphi}_{\ell^2(\bbZ)}^2 = \int \abs{\wti\varphi(k)}^2 \rho(k)\, dk
\end{equation}
Moreover, for any $\wti\varphi\in L^2 ((-\delta,\delta),dk)$, \eqref{4.31} defines a $\varphi\in\ran(P_{[b,b+\veps)}(J_0))$.

\item[{\rm{(iii)}}]
\begin{equation} \lb{4.33}
0 < \inf_{k\in (-\delta,\delta)} \rho(k) = \rho_- < \sup_{k\in (-\delta,\delta)} \rho(k) = \rho_+ <\infty
\end{equation}

\item[{\rm{(iv)}}] $E(k)=E(-k)$ and maps $[0,\delta)$ to $[0,\veps)$. For some $c_1 >0$, we have
\begin{equation} \lb{4.34}
E(k) \geq b+c_1 k^2
\end{equation}

\item[{\rm{(v)}}]
\begin{equation} \lb{4.35}
\sup_{\substack{k\in (-\delta,\delta) \\ n\in\bbZ}} \abs{u_n (k)} = c_2 <\infty
\end{equation}

\item[{\rm{(vi)}}] If
\begin{equation} \lb{4.36}
v_n(k) = e^{-i\theta(k)n} u_n(k)
\end{equation}
then for some $c_3 <\infty$ and all $n\in\bbZ$,
\begin{equation} \lb{4.37}
\abs{v_n(k) - v_n(0)} \leq c_3 k^2
\end{equation}

\item[{\rm{(vii)}}] $\theta$ is $C^2$ on $(-\delta,\delta)$ and
\begin{equation} \lb{4.38}
\inf_{x\in (-\delta,\delta)} \theta'(k) >0
\end{equation}

\item[{\rm{(viii)}}]
\begin{equation} \lb{4.39}
E(-k) = E(k), \quad u_n (-k) = \ol{u_n(k)}, \quad \theta(-k) =-\theta(k), \quad \rho(-k)=\rho(k)
\end{equation}
\end{SL}

Then for some $C$ and all $\delta J$, we have
\begin{equation} \lb{4.40}
\sum_{\substack{e\in\sigma_\di (J_0 + \delta J) \\ e\in (a,b)}} (b-e)^{1/2} \leq C
\sum_{n=-\infty}^\infty \, \abs{\delta a_n} + \abs{\delta b_n}
\end{equation}
\end{theorem}

The analog of $(H_0-E)^{-1/2}$ bounded from $L^2$ to $L^\infty$ is missing since $\ell^2 \subset\ell^\infty$, and thus
\begin{equation} \lb{4.41}
\norm{(J_0-E_0)^{-1/2} f}_\infty \leq \dist (E_0,\sigma(J_0))^{-1/2} \norm{f}_2
\end{equation}
With this remark and the bound of \cite{HS}, the proof is identical to that of Theorem~\ref{T4.1} if we use an
additional argument. Following \cite{HS}, we define $\delta J_\pm$ to be the Jacobi matrices with parameters
\begin{align}
\delta b_n^\pm &= \max\{0,\pm b_n\} + \tfrac12\, a_n + \tfrac12\, a_{n+1} \lb{4.26} \\
\delta a_n^\pm &= \pm \tfrac12\, a_n \lb{4.27}
\end{align}
so $\delta J_\pm \geq 0$ as matrices, $\delta J=\delta J_+ - \delta J_-$, and
\begin{equation} \lb{4.28}
\norm{(\delta J_\pm)^{1/2}}_{\text{\rm{HS}}}^2 = \tr(\delta J_\pm) \leq \sum_n \, \abs{b_n} + 2a_n
\end{equation}

Finally, we need to say something about the sum over eigenvalues on semi-infinite intervals but a distance $1$ from
$\sigma(H_0)$ or $\sigma(J_0)$ (since Theorems~\ref{T4.1} and \ref{T4.4} control the sum of $(\inf \sigma(J_0)-1,
\inf\sigma(H_0))$, and similarly for $J_0$). We discuss the discrete case first.

\begin{proposition}\lb{P4.5} Let $A$ be a bounded operator on a Hilbert space and $B$ trace class with $\alpha =
\inf\sigma(A)$. Then
\begin{equation} \lb{4.42}
\sum_{\substack{e\in\sigma_\di (A+B) \\ e\leq \alpha -1}} (\alpha-e)^{1/2} \leq \tr(\abs{B})
\end{equation}
\end{proposition}

\begin{proof} Let $\{e_n\}_{n=1}^\infty$ be a counting of the eigenvalues in $(-\infty, \alpha-1)$ and
$\{\varphi_n\}_{n=1}^\infty$ the eigenvectors. Then, since $\alpha-e_n\geq 1$,
\begin{align}
\sum_{n=1}^\infty (\alpha-e_n)^{1/2}
&\leq \sum_{n=1}^\infty (\alpha - e_n) \notag \\
&\leq \sum_{n=1}^\infty (\varphi_n, (\alpha-A)\varphi_n) - (\varphi_n, B\varphi_n) \notag \\
&\leq \sum_{n=1}^\infty (\varphi_n, B_- \varphi_n) \lb{4.43} \\
&\leq \tr(\abs{B})
\end{align}
where \eqref{4.43} comes from $A\geq \alpha$.
\end{proof}

\begin{proposition} \lb{P4.6} Let $h_0=-\f{d^2}{dx^2}$ on $L^2(\bbR,dx)$. Let $H_0$ be an operator for which, for some
$\gamma >0$,
\begin{equation} \lb{4.44}
H_0\geq\gamma h_0 + \beta
\end{equation}
Let $\alpha = \inf\sigma(H_0)$. Then there exists $C_1,C_2 >0$ so that for all $V\in L^1$,
\begin{equation} \lb{4.45}
\sum_{\substack{ e\in\sigma_\di (H_0+V) \\ e<\alpha -C_1}} (\alpha -e)^{1/2} \leq C_2 \int \abs{V(x)}\, dx
\end{equation}
\end{proposition}

\begin{proof} By \eqref{4.44}, $\beta\leq\alpha$. Let $e<\beta$. Then, by \eqref{4.44},
\begin{align*}
N(H_0+V\leq e) &\leq N(\gamma h_0 +V\leq e-\beta) \\
&= N(h_0 + \gamma^{-1} V\leq \gamma^{-1} (e-\beta))
\end{align*}
so using the critical Lieb--Thirring bound for $h_0$,
\begin{equation} \lb{4.46}
\sum_{e<\beta}  \sqrt{\gamma^{-1} (\beta-e)} \leq \tfrac12\, \gamma^{-1} \int V(x)\, dx
\end{equation}
If $e<\beta-1$, then $\alpha -e\leq (\beta-e)(\alpha-\beta+1)$, so
\[
\sum_{e<\beta-1}  \sqrt{\alpha-e}\leq \tfrac12\, (\alpha-\beta+1)^{1/2} \gamma^{-1/2} \int
\abs{V(x)}\, dx
\qedhere
\]
\end{proof}

\section{One-Dimensional Periodic Schr\"odinger Operators} \lb{s5}

In this section, we prove Theorem~\ref{T1.1}, that is, prove critical Lieb--Thirring bounds in individual gaps
for perturbations of periodic Schr\"odinger operators. So $h_0=-\f{d^2}{dx^2}$ on $L^2 (\bbR,dx)$ and $V_0$ is
a periodic potential with
\begin{equation} \lb{5.1}
V_0(x+2\pi) = V_0(x)
\end{equation}
(there is no loss with picking the period to be $2\pi$). We suppose
\begin{equation} \lb{5.2}
\int_{-\pi}^\pi \abs{V_0(x)}\, dx < \infty
\end{equation}

Then, by a Sobolev estimate, $V_0$ is a form-bounded perturbation of $h_0$ with relative bound zero. Thus, $H_0=
h_0 + V_0$ is a well-defined form sum, and if $E_0 <\inf \sigma(H_0)$, then $(h_0+1)^{1/2} (H_0-E_0)^{-1/2}$ is
bounded from $L^2$ to $L^2$. So by a Sobolev estimate, $(H_0-E_0)^{-1/2}$ is bounded from $L^2$ to $L^\infty$, that is,
(ii) of Theorem~\ref{T4.1} is valid.

The following facts are well known (see \cite[Sect.~XIII.16]{RS4} which supposes $V_0$ bounded, but no changes are
needed to handle the locally $L^1$ case; see also \cite{MW}):
\begin{SL}
\item[(i)] If $U\colon L^2(\bbR,dx)\to L^2([0,2\pi),L^2([0,2\pi],dx); \f{d\varphi}{2\pi})$ is defined by
\begin{equation} \lb{5.3}
(Uf)_\varphi(x)=\sum_{n=-\infty}^\infty e^{-i\varphi n} f(x+2\pi n)
\end{equation}
then $U$ is unitary.

\item[(ii)] If $h_0 (\varphi)$ is defined for $\varphi\in [0,2\pi)$ on $L^2 ([0,2\pi],dx)$ as $-\f{d^2}{dx^2}$ with
boundary conditions
\begin{equation} \lb{5.4}
u(2\pi) = e^{i\varphi}u(0) \qquad u'(2\pi) = e^{i\varphi} u'(0)
\end{equation}
and $H(\varphi) = h_0(\varphi) + V_0$, then
\begin{equation} \lb{5.5}
U\!H\!U^{-1} g_\varphi = H(\varphi)g_\varphi
\end{equation}

\item[(iii)] Each $H(\varphi)$ has compact resolvent and so eigenvalues $\{\veps_j(\varphi)\}_{j=1}^\infty$ and
eigenvectors $u_j^{(\varphi)}(x)$ so that
\begin{equation} \lb{5.6}
H(\varphi) u_j^{(\varphi)} = \veps_j^{(\varphi)} u_j^{(\varphi)}
\end{equation}
If, for $x\in [0,2\pi)$,
\begin{equation} \lb{5.7}
v_j^{(\varphi)}(x) = e^{-i \varphi x/2\pi} u_j^{(\varphi)} (x)
\end{equation}
then, by \eqref{5.4}, $v_j$ has a periodic extension and all $v_j^{(\varphi)}$ lie in $Q(h_0(\varphi\equiv 0))$
and obey (where $p=-id/dx$)
\begin{equation} \lb{5.8}
\biggl[ h_0(0) +2 \f{\varphi}{2\pi} \, p + \biggl( \f{\varphi}{2\pi}\biggr)^2 + V_0\biggr] v_j^{(\varphi)}
=\veps_j (\varphi) v_j^{(\varphi)}
\end{equation}
If the operator in $[\dots]$ in \eqref{5.8} is $\wti H(\varphi)$, then it is a Kato analytic family of type (B).
Moreover, for any single $j$, $v_j\in Q(h_0)$ with bounded norm, by a Sobolev estimate,
\begin{equation} \lb{5.8a}
\sup_{\varphi, x} \abs{v_j^{(\varphi)}(x)} <\infty
\end{equation}
for each fixed $j$.

\item[(iv)] $\veps_j (2\pi -\varphi) = \veps_j(\varphi)$ and $v_j^{(2\pi -\varphi)} = \ol{v_j^{(\varphi)}}$. On
$[0,\pi]$, $(-1)^{j+1}\veps_j$ is strictly monotone increasing, so $\veps_1(0) < \veps_1(\pi)\leq \veps_2(\pi) <
\veps_2(0) \leq \veps_3(0)\leq \cdots < \veps_{2j-1}(\pi) \leq \veps_{2j}(\pi) < \veps_{2j}(0) \leq \veps_{2j+1}(0)
\dots$. The gaps in $\spec(H)$ are exactly the nonempty $(\veps_{2j-1}(\pi), \veps_{2j}(\pi))$ and
$(\veps_{2j}(0),\veps_{2j+1}(0))$. If such a gap is nonempty, we say it is an open gap.

\item[(v)] There is an entire analytic function $\Delta(E)$ so that
\begin{equation} \lb{5.9}
\Delta(\veps_j(\varphi)) =2 \cos(\varphi)
\end{equation}
and a gap is open if and only if $\Delta'(\veps)\neq 0$ at the endpoints of the gap. It then follows from \eqref{5.9}
that at an open gap,
\begin{equation} \lb{5.10}
\veps'_j(0 \text{ or } \pi) =0 \qquad \veps''_j (0 \text{ or } \pi) \neq 0
\end{equation}
\end{SL}

\smallskip
This says that the framework of Theorem~\ref{T4.1} is applicable. For notational simplicity, we consider an open gap at
$\varphi=0$ (below, if $\varphi =\pi$, replace $k=\varphi/2\pi$ by $k=(\varphi-\pi)/2\pi$ and the associated $v_j$ is
then antiperiodic) and the top end of the gap at energy $b = \veps_n(0)$. We take $\delta=1/4$, $k=\varphi/2\pi$,
and $\theta(k)=k$. $E(k)=\veps_n (2\pi k)$. For $0\leq x < 2\pi$,
\begin{equation} \lb{5.11}
u(x + 2\pi m, k) = u_n^{(2\pi k)} (x) e^{2\pi imk}
\end{equation}
using the boundary condition \eqref{5.4}.  We set $\veps=E(\f14)=\veps_n(\f{\pi}{2})$. $\ran(P_{[b,b+\veps]}(H_0))$ is
exactly those $f$ with $(Uf)_\varphi = 0$ if $\varphi\notin (-\f{\pi}{2}, \f{\pi}{2})$ and equal to a multiple of
$u_n^{(\varphi)}$ if $\varphi\in (-\f{\pi}{2},\f{\pi}{2})$
\[
\ti f(k) = \jap{(Uf)_{(\varphi=2\pi k)}, u_n^{(2\pi k)}}
\]

\eqref{4.4} holds with $\rho(k)\equiv 1$, so \eqref{4.5} is immediate. \eqref{4.6} holds by the fact that $\veps$ is real
analytic on $(-\pi,\pi)$ and that \eqref{5.10} holds. \eqref{4.7} holds by \eqref{5.8a}.

\eqref{4.9} holds because $v$ is periodic in $x$ and $u$ is real analytic in $k$ with $\f{du}{dk}=0$. (viii) and
(ix) are immediate.

Theorem~\ref{T4.1} thus implies Theorem~\ref{T1.1}.

\smallskip
We have only controlled individual gaps. It is natural to ask if one can sum over all the typically infinitely many gaps.
We believe this will be difficult with our methods. The issue involves the constant $c_3$ in \eqref{4.9}. For large $n$,
the $n$-th band has size $O(n)$ near an energy of $O(n^2)$. The size $g_n$ of the $n$-th gap is small. If $v_0$ is $C^\infty$,
it is known (Hochstadt \cite{Hoch}) that $g_n=o(n^k)$ for all $k$; and for $v_0(x)=\lambda\cos (x)$, it is known
(\cite{AS}) that $g_n\sim n^{-2n}$. Away from $k=0$ or $\pi$, $\veps'_n(k)\sim n$ and it goes from $\veps'_n=0$ to
$n$ in a distance of size $O(g_n)$, that is, we expect $\veps''_n(0)\sim g_n^{-1}n$. Thus, we expect $c_3$ to be
$O(ng_n^{-1})$. While $c_3$ is divided by $c_1$, which is also large, $c_3\sim\sup_{\abs{k}\leq \delta} \veps''_n(k)$,
while $c_1\sim\inf_{\abs{k}\leq\delta} \veps''(k)$. So unless we take $\delta\downarrow 0$ (which itself causes
difficulties), the cancellation will only be partial. Thus, we have not been able to sum over all gaps.

\section{Critical Lieb--Thirring Bounds and Generalized Nevai Conjecture
for Finite Gap Jacobi Matrices} \lb{s6}

In this section, we turn to perturbations of elements of the isospectral torus of Jacobi matrices assigned to a finite
gap set, $\fre$, as described in the introduction. Our main goal is:

\begin{theorem} \lb{T6.1} Let $\fre$ be a finite gap set and $(\beta_j, \alpha_{j+1})$ a gap in $\bbR\setminus\fre$.
Let $\{a_n^{(0)}, b_n^{(0)}\}_{n=-\infty}^\infty$ be an element of the isospectral torus. Then for a constant $C$ and any
$\{a_n,b_n\}_{n=1}^\infty$ a set of Jacobi parameters obeying the two-sided analog of \eqref{1.11},
\begin{equation} \lb{6.1}
\sum_{e\in (\beta_j, \alpha_{j+1})\cap\sigma(J)} \dist(e,\sigma_\e(J))^{1/2}
\leq C \sum_{n=-\infty}^\infty \, \abs{a_n-a_n^{(0)}} + \abs{b_n - b_n^{(0)}}
\end{equation}
\end{theorem}

\begin{remarks} 1. The proof shows $C$ can be chosen independently of the point on the isospectral torus of $\fre$.

\smallskip
2. The proof works on $(\alpha_1-1,\alpha_1)$ and $(\beta_{\ell+1}, \beta_{j+1}+1)$ and then, using Proposition~\ref{P4.5},
one gets bounds for $e\in (-\infty, \alpha_1)\cap\sigma(J)$ and for $e\in (\beta_{\ell+1},\sigma)\cap\sigma(J)$, and
then since there are finitely many gaps:
\end{remarks}

\begin{corollary} \lb{C6.2} Under the hypotheses of Theorem~\ref{T6.1},
\begin{equation} \lb{6.2}
\sum_{e\in\sigma(J)} \dist(e,\sigma_\e(J))^{1/2} \leq \text{RHS of \eqref{6.1}}
\end{equation}
\end{corollary}

This then implies

\begin{proof}[Proof of Theorem~\ref{T1.3}] Christiansen, Simon, and Zinchenko \cite[Thm.~4.5]{CSZ2} prove that
\eqref{1.12} is implied by
\begin{alignat}{2}
&\text{(a)} \qquad && \text{LHS of \eqref{6.2}}  <\infty \lb{6.3} \\
&\text{(b)} \qquad && \lim\biggl( \f{a_1 \dots a_n}{C(\fre)^n}\biggr) \text{ exists in } (0,\infty) \lb{6.4}
\end{alignat}
\eqref{6.3} follows from \eqref{1.11}, \eqref{6.2}, and an eigenvalue interlacing argument (since \eqref{6.2} is for
full-line operators). (b) is immediate from $\sum_{n=1}^\infty \abs{a_n-a_n^{(0)}}<\infty$ and the analog of
\eqref{6.4} for $a_j^{(0)}$ (see \cite[Cor.~7.4]{CSZ2}).
\end{proof}

We will prove Theorem~\ref{T6.1} by showing the applicability of our Theorem~\ref{T4.4}. This will require the theory
of eigenfunction expansions for one-dimensional a.c.\ reflectionless systems and the theory of Jost functions for
finite gap operators, where we'll follow the presentations of Breuer--Ryckman--Simon \cite{BRS} and
Christiansen--Simon--Zinchenko \cite{CSZ1}, respectively. We'll use their theorems but not their precise notation
since there are conflicts between our notation in Section~\ref{s4} and theirs.

We'll use $U_n^\pm (\lambda)$ for the Weyl solutions of \cite{BRS} at energy $\lambda$, defined for Lebesgue a.e.\
$\lambda\in\sigma(J^{(0)})$. They obey $J^{(0)} U^\pm = \lambda U^\pm$ and are normalized by
\begin{equation} \lb{6.5}
U_0^\pm (\lambda) =1
\end{equation}
Since $J_0$ is reflectionless (see \cite{Rice}), we have
\begin{equation} \lb{6.6}
U_n^- = \ol{U_n^+}
\end{equation}
so the functions $f_\pm(\lambda)$ of \cite[eqn~(2.4)]{BRS} are equal with
\begin{equation} \lb{6.7}
f_\pm (\lambda) = -(4\pi a_0)^{-1} (\Ima U_1^+(\lambda))^{-1}
\end{equation}
which we call $f$ below. Theorem~2.2 of \cite{BRS} implies that if (for $\varphi_n\in\ell^1\cap\ell^\infty$)
\begin{equation} \lb{6.8}
\widehat\varphi_\pm (\lambda) = \sum_n\, \ol{U_n^\pm(\lambda)}\, \varphi_n
\end{equation}
then
\begin{gather}
\varphi_n = \int [\widehat\varphi_+(\lambda) U_n^+ (\lambda) + \widehat\varphi_- (\lambda) U_n^-(\lambda)]
f(\lambda)\, d\lambda \lb{6.9} \\
\widehat{J\varphi_\pm} (\lambda) = \lambda\, \widehat\varphi_\pm (\lambda) \lb{6.10} \\
\norm{P_{a,b}(J_0)\varphi}^2 = \int_a^b (\abs{\widehat\varphi_+(\lambda)}^2 + \abs{\widehat\varphi_-(\lambda)}^2)
f(\lambda)\, d\lambda \lb{6.11}
\end{gather}

From \cite{CSZ1}, we need the covering map $\x\colon\bbC\cup\{\infty\}\setminus\calL\to\calS$, where $\calS$ is the
two-sheeted compact Riemann surface associated to the function
\begin{equation} \lb{6.11x}
D(x) =\biggl(\, \prod_{j=1}^{\ell+1} (x-\alpha_j)(x-\beta_j)\biggr)^{1/2}
\end{equation}
$\calL$, the limit set of a certain Fuchsian group, is a closed, nowhere dense, perfect subset of $\partial\bbD=\{z\mid
\abs{z}=1\}$. There is an open subset, $\calF\subset\bbD$ on which $\x$ is one-one to $\bbC\cup\{\infty\}\setminus\fre$,
whose closure is a fundamental domain for the Fuchsian group. For any band, $[\beta_j,\alpha_{j+1}]$, in $\bbR\setminus\fre$,
there are $e^{i\varphi_0}, e^{i\varphi_1}\in\partial\bbD$ with $\varphi_0 <\varphi_1$, so $\varphi\mapsto \x(e^{i\varphi})$
maps $(\varphi_0,\varphi_1)$ bijectively onto the upper lip of the cut $(\beta_j,\alpha_{j+1})$. What is crucial for
us is that
\begin{equation} \lb{6.12}
\f{\partial\x(e^{i\varphi})}{\partial\varphi} \neq 0, \quad \varphi\in (\varphi_0,\varphi_1); \qquad
\f{\partial \x}{\partial\varphi} =0, \quad \f{\partial^2 \x}{\partial\varphi^2}\neq 0
\quad\text{at $\varphi_0$ or $\varphi_1$}
\end{equation}
$\x$ is analytic in a neighborhood of $\{e^{i\varphi}\mid \varphi\in (\varphi_0,\varphi_1)\}$.

The fundamental Blaschke function, $B$, associated to $\x$ is a meromorphic function on $\bbC\cup\{\infty\}\setminus\calL$,
which is a Blaschke product, and so obeys
\begin{equation} \lb{6.13}
\abs{z} <1 \Rightarrow \abs{B(z)} < 1 \qquad \abs{z}=1\Rightarrow \abs{B(z)}=1
\end{equation}
This, in turn, implies on $\partial\bbD\setminus\calL$,
\begin{equation} \lb{6.14}
B(e^{i\varphi}) = e^{i\ti\theta(\varphi)} \qquad \f{\partial\ti\theta}{\partial\varphi} >0
\end{equation}
and $\wti\theta$ is real analytic on $\partial\bbD\setminus\calL$.

We will let $\delta <\varphi_1 - \varphi_0$ and define, for $k\in(-\delta,\delta)$,
\begin{equation} \lb{6.15}
E(k) = \beta_j + \x(e^{i(\varphi_0 + k)})
\end{equation}
for $k\geq 0$ and $E(k)$ even. It is real analytic on $(-\delta,\delta)$ by \eqref{6.12}. Define
\begin{equation} \lb{6.16}
\theta(k) = \begin{cases}
\ti\theta(k+\varphi_1) - \ti\theta(\varphi_1) & k>0 \\
-\theta(-k) & k<0
\end{cases}
\end{equation}
which is $C^\infty$ in $k$.

We let $\bbG$ denote the isospectral torus. There is a real analytic map $T\colon\bbG\to\bbG$ and a coordinate system
on $\bbG$ in which $T$ is a group translation, and functions $A,B$ on $\bbG$ so that
\begin{equation} \lb{6.17}
a_n (\vy) = A(T^n \vy) \qquad b_n(\vy) = B(T^n\vy)
\end{equation}
for the Jacobi parameters for the Jacobi matrix $J^{(\vy)}$ with $\vy$ in $\bbG$.

There are functions $\J(z;\vy)$ (the Jost function) for $z\in\bbC\cup\{\infty\}\setminus\calL$, $\vy\in\bbG$ which are
meromorphic in $z$, real analytic in $\vy$, and whose only poles lie in $\bbC\cup\{\infty\}\setminus\ol{\bbD}$ with
limit points only in $\calL$. In particular, $\J$ is analytic, uniformly in $\vy$, for $z$ in a neighborhood of
$\{e^{i\varphi}\mid \varphi\in [\varphi_0,\varphi_1]\}$. The Jost solution is given by
\begin{equation} \lb{6.18}
\J_n (z;y) = a_n (y)^{-1} B(z)^n \J(z;T^n(\vy))
\end{equation}

Suppose, for now, that the original Jacobi matrix, $J^{(0)}$, corresponding to $\vy=0$, has
\begin{equation} \lb{6.19}
\J_{n=0} (e^{i\varphi_0}; \vy=0)\neq 0
\end{equation}
(equivalently, $\J(e^{i\varphi_0}; \vy=0)\neq 0$). $\J_n$ solves the difference equation $J^{(\vy)}\J_n(z;y) =
x(z) \J_n(z;y)$, so to get the normalization condition \eqref{6.5}, we have
\begin{equation} \lb{6.20}
U_n^+ (\lambda) = \f{\J_n (\z(\lambda); \vy=0)}{\J_0 (\z(\lambda); \vy=0)}
\end{equation}
where $\z(\lambda)$ is determined by $\x(\z(\lambda)) =\lambda$ with $\z(\lambda)\in \{e^{i\varphi}\mid\varphi_0
\leq \varphi\leq \varphi_1\}$.

We define $\rho(k)$ by
\begin{equation} \lb{6.21}
\rho(k) = \begin{cases}
f(E(k))\, \f{d}{dk} \, \x(e^{i(\varphi_0 +k)}) & k\geq 0 \\
\rho(-k) & k<0
\end{cases}
\end{equation}

We define $u_n^+(k)$ for $k\in (-\delta,\delta)$ by
\begin{equation} \lb{6.22}
u_n^+(k) = \begin{cases}
U_n^+ (E(k)) \rho(k) & k>0 \\
\ol{U_n^+} \,(E(k)) \rho(k) & k<0
\end{cases}
\end{equation}

Finally,
\begin{equation} \lb{6.23}
\wti\varphi(k) = \begin{cases}
\widehat\varphi_+(E(k)) & k\geq 0 \\
\widehat\varphi_- (E(k)) & k<0
\end{cases}
\end{equation}

$\rho$ is picked to turn $f(\lambda)\, d\lambda$ in \eqref{6.11} to $\rho(k)\, dk$. It is then straightforward to check
that \eqref{4.30} and \eqref{4.32} hold. Away from $k=0$, $\rho(k)$ is smooth, bounded, and nonvanishing. Since $u_j^+ 
(k=0)=0$, $\Ima u_1^+ (k=0)=0$ and $f$ blows up there, but exactly as $1/k[\left.\theta'(k)\right|_{k=0}]$. Since
$\f{\partial \x}{\partial k}$ vanishes as $k$, by \eqref{6.12}, $\rho$ has a smooth nonzero limit as $k\downarrow 0$,
that is, \eqref{4.33} holds.

The relation \eqref{6.12} shows that at $k=0$, $E'(k)=0$, $E''(k)\neq 0$, so \eqref{4.34} holds. Since $\J$ is uniformly
bounded on $\bbG$ when $z\in \{e^{i\varphi}\mid\varphi_0\leq\varphi \leq\varphi_1\}$, \eqref{4.35} follows from \eqref{6.18}.

$\theta$ is defined so the $B(z)^n$ in \eqref{4.18} is replaced by $B(e^{i\varphi_0})^n$ in the formula for $v$.
Thus, $k$ derivatives are derivatives of $\J(e^{i(\varphi_0 + k)}, T^n (\vy =0))$ which are bounded uniformly  in $n$
by compactness of $\bbG$. First derivatives are zero and second derivatives are uniformly bounded in $n$ and $k\in
(0,\delta)$, so \eqref{4.37} holds. \eqref{4.36} follows from \eqref{6.14}. Thus, if \eqref{6.19} holds, Theorem~\ref{T4.4}
is applicable and proves Theorem~\ref{T6.1}.

Since nonzero solutions of a Jacobi eigenfunction equation cannot vanish at two successive points, if \eqref{6.19} fails
for $\{a_n^{(0)}, b_n^{(0)}\}_{n=-\infty}^\infty$, it will not for $\{a_{n+1}^{(0)}, b_{n+1}^{(0)}\}_{n=-\infty}^\infty$,
so we get Theorem~\ref{T6.1} for a translated $J^{(0)}$. But since the conclusions are translation invariant, the theorem
for the translated $J^{(0)}$ implies it for the original $J^{(0)}$.

Using the extensive literature on finite gap continuum Schr\"odinger operators (see Gesztesy--Holden \cite{GH} and
references therein), it should be possible to prove a continuum analog of the results of this section.

\section{Dirac Equations} \lb{s7}

Our decoupling results in Section~\ref{s2} allow us to obtain some bounds on eigenvalues in the gap of one-dimensional
Dirac operators. We will not require the results of Section~\ref{s4}. Let $\sigma_1,\sigma_3$ be the standard Pauli
matrices, $\sigma_1 = (\begin{smallmatrix} 0 & 1 \\ 1& 0\end{smallmatrix})$, $\sigma_3 =(\begin{smallmatrix}1 & 0 \\ 0 & -1
\end{smallmatrix})$, $p=\f{1}{i}\, \f{d}{dx}$ on $L^2 (\bbR,dx)$, and
\begin{equation} \lb{7.1}
D_0 = p\sigma_1 + m\sigma_3 =
\left( \begin{array}{rr}
m & p \\
p & -m
\end{array} \right)
\end{equation}
be the free Dirac operator on $L^2 (\bbR,\bbC^2;dx)$. Here we'll prove

\begin{theorem}\lb{T7.1} Let $\gamma \geq \f12$ and $V\in L^{\gamma+1/2}(\bbR,dx)\cap L^{\gamma +1}(\bbR,dx)$. If
$E_j$ denotes the eigenvalues of $D_0 +V$ in the gap $(-m,m)$, counting multiplicities, then
\begin{equation} \lb{7.2}
\sum_j (m-\abs{E_j})^\gamma \leq C_{1,\gamma} \int_\bbR \abs{V(x)}^{\gamma + 1}\, dx + C_{2,\gamma} \sqrt{m}
\int_\bbR \abs{V(x)}^{\gamma + 1/2}\, dx
\end{equation}
for some constants $C_{1,\gamma}, C_{2,\gamma}$ independent of $V$\! and $m$.
\end{theorem}

The proof below yields explicit values of the constants.

The idea of the proof is to use Theorem~\ref{T1.4} to reduce bounds to the scalar operators $\sqrt{p^2+m^2}-m-V_\pm$,
and then to use Lieb--Thirring inequalities for $p^2 - V_\pm$ and for $\abs{p}-V_\pm$ to control $\sqrt{p^2 + m^2}
-m-V_\pm$.

\begin{theorem}\lb{T7.2} Let $\gamma >0$ and $V\in L^{\gamma + 1/2} (\bbR,dx) \cap L^{\gamma+1}(\bbR)$. If $E_j$ denotes
the eigenvalues of $D_0 + V$ in $(-m,m)$, then
\begin{equation} \lb{7.3}
\sum_j (m-\abs{E_j})^\gamma \leq 2 [S_\gamma (H_0-V_-) + S_\gamma (H_0 - V_+)]
\end{equation}
where $H_0$ is the operator $\sqrt{p^2 + m^2} - m$ on $L^2 (\bbR,dx)$.
\end{theorem}

We emphasize that we consider the operator $H_0$ acting on {\it spinless} (i.e., scalar) functions. One might wonder
whether the inequality is true without the factor of $2$.

\begin{proof} By Theorem~\ref{T1.4} and \eqref{4.14}, one has
\begin{align}
\sum_j (m-\abs{E_j})^\gamma
&= \gamma \int_0^m (m-E)^{\gamma-1} N(D_0 + V\in (-E,E))\, dE \notag \\
&\leq \gamma \int_0^m (m-E)^{\gamma-1} (N(V_-^{1/2} (D_0-E)^{-1} V_-^{1/2} >1) \notag \\
&\qquad \qquad + N(V_+^{1/2} (D_0+E)^{-1} V_+^{1/2} <-1))\, dE \lb{7.4x}
\end{align}
The $2\times 2$ matrix, $(\begin{smallmatrix} m-E & p \\ p & -m-E\end{smallmatrix})$, has eigenvalues $-E\pm
\sqrt{p^2+ m^2}$, which implies the operator inequalities
\[
\mp(D_0\pm E)^{-1} \leq (H_0 + m-E)^{-1} \otimes I
\]
Using this and the Birman--Schwinger principle, we find that
\begin{align*}
N(V_-^{1/2} (D_0-E)^{-1} V_-^{1/2} >1 )
&\leq 2N(V_-^{1/2} (H_0 + m-E)^{-1} V_-^{1/2} >1) \\
&= 2N(H_0 -V_- < -m+E)
\end{align*}
and
\begin{align*}
N(V_+^{1/2} (D_0 +E)^{-1} V_+^{1/2} < -1)
&\leq 2N(V_+^{1/2} (H_0 + m-E)^{-1} V_+^{1/2} >1 ) \\
&= 2N(H_0-V_+ <-m+E)
\end{align*}
Plugging this into \eqref{7.4x} and changing variables $\tau = m-E$, we obtain
\[
\sum_j (m-\abs{E_j})^\gamma \leq 2\gamma \int_0^m \tau^{\gamma-1} (N(H_0-V_- < -\tau) +
N(H_0 - V_+ <-\tau))\, d\tau
\]
Extending the integration to the whole interval $(0,\infty)$, we obtain \eqref{7.3}.
\end{proof}

Theorem~\ref{T7.1} follows immediately from Theorem~\ref{T7.2} and Proposition~\ref{P7.3} below. It will rely on
classical Lieb--Thirring bounds for $p^2 +V$ and those for $\abs{p}+V$  in the following form (see Remark~4
on page~517 of \cite{Daub} or eqn.~(13) in \cite{FLS}):
\begin{alignat}{2}
S_\gamma (p^2 + V) &\leq L_\gamma \int_\bbR V(x)_-^{\gamma + 1/2}\, dx &&\qquad \gamma \geq \tfrac12 \lb{7.4} \\
S_\gamma (\abs{p}+V) &\leq \ti L_\gamma \int_\bbR V(x)_-^{\gamma+1}\, dx && \qquad \gamma >0 \lb{7.5}
\end{alignat}

\begin{proposition} \lb{P7.3} Let $\gamma \geq \f12$ and let $0\leq W \in L^{\gamma+1/2}(\bbR,dx)
\cap L^{\gamma+1}(\bbR)$. Then
\begin{equation} \lb{7.6}
S_\gamma (H_0-W) \leq C_{1,\gamma} \int_\bbR W(x)^{\gamma+1}\, dx + C_{2,\gamma} \sqrt{m}
\int_\bbR W(x)^{\gamma + 1/2}\, dx
\end{equation}
for some constants $C_{1,\gamma}, C_{2,\gamma}$ independent of $W$\! and $m$.
\end{proposition}

\begin{remark} One could replace the right side of \eqref{7.6} by a phase space bound.
\end{remark}

\begin{proof} Using the Birman--Schwinger principle, we write
\begin{align}
S_\gamma (H_0-W) &= \gamma \int_0^\infty N(H_0-W \leq -\tau) \tau^\gamma\, d\tau  \notag \\
&=  \gamma \int_0^\infty N(W^{1/2} (H_0-\tau)^{-1} W^{1/2} >1) \tau^\gamma\, d\tau \lb{7.6x}
\end{align}
In order to estimate $N(W^{1/2} (H_0-\tau)^{-1} W^{1/2} >1)$, we fix two parameters, $0<\theta <1$
and $\rho >0$, and denote by $P$ and $P^\perp$ the spectral projections of $H$ onto the intervals
$[0,\rho m)$ and $[m\rho,\infty)$, respectively. By Proposition~\ref{P2.1},
\begin{equation} \lb{7.7}
\begin{split}
N(W^{1/2} (H_0-\tau)^{-1} W^{1/2} >1) &\leq N(W^{1/2} P(H_0-\tau)^{-1} W^{1/2} >\theta) \\
&\qquad N(W^{1/2} P^\perp (H_0-\tau)^{-1} W^{1/2} >1-\theta)
\end{split}
\end{equation}

There are constants, $c_1, c_2 >0$, depending on $\rho$ such that
\begin{alignat}{2}
\sqrt{p^2+m^2} - m & \geq \f{c_1}{m}\, p^2 && \qquad \text{if } \abs{p}\leq \rho m \lb{7.8} \\
\sqrt{p^2+m^2} - m & \geq c_2 \abs{p} && \qquad \text{if } p\geq \rho m \lb{7.9}
\end{alignat}
Indeed, one can choose
\begin{equation} \lb{7.10}
c_1 = \f{\sqrt{\rho^2+1} -1}{\rho^2} \qquad c_2 = \f{\sqrt{\rho^2 +1}-1}{\rho}
\end{equation}
This and the Birman--Schwinger principle yield
\begin{align}
N(W^{1/2} P(H_0-\tau)^{-1} W^{1/2} >\theta )
&\leq N \biggl(W^{1/2} \biggl( \f{c_1 p^2}{m} - \tau\biggr)^{-1} W^{1/2} >\theta\biggr) \notag \\
&= N\biggl( \f{c_1 p^2}{m} -\theta^{-1} W<-\tau\biggr) \lb{7.11}
\end{align}
and
\begin{align*}
N(W^{1/2} P^\perp (H-\tau)^{-1} W^{1/2} >1-\theta)
&\leq N(W^{1/2} (c_2 \abs{p}-\tau)^{-1} W^{1/2} >1-\theta) \\
&= N(c_2 \abs{p}-(1-\theta)^{-1} W <-\tau)
\end{align*}
Plugging this into \eqref{7.6x} and doing the $\tau$-integration, we arrive at
\[
S_\gamma (H_0-w) \leq S_\gamma \biggl( \f{c_1 p^2}{m}-\theta^{-1} W\biggr) +
S_\gamma (c_2 \abs{p} -(1-\theta)^{-1} W)
\]
Using \eqref{7.4} and \eqref{7.5}, we get
\[
\begin{split}
S_\gamma (H_0-W) & \leq c_1^{-1/2} \theta^{-\gamma - 1/2} L_\gamma \sqrt{m} \int W^{\gamma+1/2}\, dx  \\
&\qquad \quad +  c_2^{-1} (1-\theta)^{-\gamma-1} \ti L_\gamma \int W^{\gamma +1}\, dx
\end{split}
\]
This completes the proof of the proposition.
\end{proof}

\section*{Appendix: Index Theory Proof of Proposition~\ref{P2.3}} \lb{App}
\renewcommand{\theequation}{A.\arabic{equation}}
\renewcommand{\thetheorem}{A.\arabic{theorem}}
\setcounter{theorem}{0}
\setcounter{equation}{0}

Here we'll provide a proof of Proposition~\ref{P2.3} using the theory of the index of a pair of orthogonal projections
from \cite{ASS}. This makes explicit the approach of Pushnitski \cite{Push} in his proofs of Proposition~\ref{P2.3}
and Theorem~\ref{T1.4}. Recall that if $P,Q$ are projections with
\begin{equation} \lb{A.1}
\dist(P-Q, \text{ compact operators}) <1
\end{equation}
(and, in particular, if $P-Q$ is compact), one can define an integer index $(P,Q)$ by the equivalent definitions:
\begin{align}
\idx(P,Q) &= \dim\ker (P-Q-1) -\dim\ker (Q-P-1) \lb{A.1a} \\
&=\dim (\ran\, P\cap \ran\, Q^\perp) -\dim\ker (\ran\, Q\cap\ran\, P^\perp) \lb{A.1b} \\
&= \text{Fredholm index of $QP$ as a map of $\ran\, P$ to $\ran\, Q$} \lb{A.1c}
\end{align}
One has \cite{ASS}:
\begin{SL}
\item[(a)] If $Q-R$ is compact, then
\begin{equation} \lb{A.2}
\idx(P,R) = \idx (P,Q) + \idx (Q,R)
\end{equation}
whenever \eqref{A.1} holds. This comes from \eqref{A.1c}, compactness of $P(Q-R)Q$ and invariance of the Fredholm index
under compact perturbations.

\item[(b)] If $P-Q$ is finite rank, then
\begin{equation} \lb{A.3}
\idx(P,Q) = \text{trace}(P-Q)
\end{equation}
and, in particular, if $P\geq Q$ also, so $\ran\, Q\subset\ran\, P$, then
\begin{equation} \lb{A.4}
\idx (P,Q) = \dim(\ran\, P\cap \ran\, Q^\perp)
\end{equation}
\item[(c)] If $Q(x)$ is norm-continuous in $x$ for $x\in [a,b]$ and $Q(x)-P$ is compact for all such $x$, then
\begin{equation} \lb{A.5}
\idx(Q(b),P) = \idx(Q(a),P)
\end{equation}
(this follows from \eqref{A.2} and $\norm{Q(x)-Q(y)} <1 \Rightarrow \idx(Q(x),Q(y))=0$).
\end{SL}

\smallskip
Let $A$ be a selfadjoint operator bounded from below and $B$ an $A$-form compact perturbation.  Then for any
$x_0$ and for $E_0$ sufficiently negative, $(A+x_0B -E_0)^{-1} - (A-E_0)^{-1}$ is compact, so by standard polynomial
approximations, $f(A+x_0B)-f(A)$ is compact for all continuous $f$ of compact support. In particular, if $E\notin
\sigma(A)\cup\sigma(A+x_0 B)$, then $P_{(-\infty,E)}(A+x_0 B)-P_{(-\infty, E)}(A)$ is compact, and so has a relative
index. Here is the key fact (a special case of eqn.\ (2.12) of Pushnitski \cite{Push}):

\begin{proposition}\lb{PA.1} Let $A$ be bounded from below and $B$ a nonnegative form compact perturbation. Suppose
$E\notin \sigma(A),\sigma(A+B)$ {\rm{(}}resp.\ $\sigma(A),\sigma(A-B)${\rm{)}}, then
\begin{equation} \lb{A.6}
\idx(P_{(-\infty,E)}(A+B), P_{(-\infty,E)}(A)) =- \delta_+(A,B;E)
\end{equation}
{\rm{(}}resp.
\begin{equation} \lb{A.7}
\idx (P_{(-\infty,E)}(A-B), P_{(-\infty,E)}(A))=\delta_- (A,B;E))
\end{equation}
\end{proposition}

\begin{proof} Since $\delta_+(A-B,B;E) = \delta_-(A,B;E)$ and $\idx(P,Q)=-\idx(Q,P)$, \eqref{A.6} implies \eqref{A.7}, so
we'll prove that.

Let $x_0\in [0,1]$ be such that $E$ is an eigenvalue of $A+x_0 B$ of multiplicity $k$. We show, for all sufficiently
small $\veps$, that
\begin{equation} \lb{A.8}
\idx(P_{(-\infty,E)}(A+(x_0 + \veps) B), P_{(-\infty,E)}(A+(x_0-\veps)B)) = -k
\end{equation}
Then, since $E$ is an eigenvalue of $A+xB$ for only finitely many $x$'s and $\idx (P_{(-\infty,E)}(A+xB),P_{(-\infty,E)}(A))$
is constant on the intervals between such $x$'s (by (c) above), \eqref{A.8} implies \eqref{A.6}.

Since $E\notin\sigma(A)$, there exists $\delta_0 >0$, so $[E-\delta_0,E+\delta_0]\cap\sigma(A)=\emptyset$, and then for all
$x$, $A+xB$ has only finitely many eigenvalues in $[E-\delta_0, E+\delta_0]$ and these eigenvalues are monotone in $x$. It
follows that we can find $\veps_0 >0$ and then $0 <\delta <\delta_0$ so that
\begin{SL}
\item[(a)] For $x\in (x_0-\veps_0, x_0 + \veps_0)$, $A+xB$ has exactly $k$ eigenvalues in $[E-\f{\delta}{2}, E+\f{\delta}{2}]$
and no eigenvalues in $[E-\delta, E -\f{\delta}{2})\cup (E + \f{\delta}{2}, E+\delta]$.

\item[(b)] If $x_0-\veps_0 < x < x_0$ (resp.\ $x_0 <x<x_0 + \veps_0)$, these $k$ eigenvalues are all in $[E-\f{\delta}{2},E]$
(resp.\ $[E,E + \f{\delta}{2}]$).
\end{SL}

\smallskip
If $0<\veps <\veps_0$, we have (the second and fourth follow from monotonicity, continuity, and (b))
\begin{gather}
P_{(-\infty, E]}(A+(x_0-\veps)B) = P_{(-\infty, E+\delta]} (A+(x_0-\veps)B) \lb{A.9} \\
\idx (P_{(-\infty, E+\delta]} (A+ (x_0-\veps)B), P_{(-\infty, E+\delta]} (A+x_0 B)) =0 \lb{A.10} \\
P_{(-\infty, E]} (A+(x_0+\veps)B) = P_{(-\infty, E-\delta]} (A+(x_0 + \veps )B) \lb{A.11} \\
\idx(P_{(-\infty, E-\delta]} (A+(x_0+\veps)B), P_{(-\infty, E-\delta]} (A+x_0 B)) =0 \lb{A.12}
\end{gather}
Thus, by \eqref{A.2},
\begin{align}
\text{LHS of \eqref{A.8}}
&= \idx(P_{(-\infty, E-\delta]} (A+x_0 B), P_{(-\infty, E+\delta]} (A+x_0 B)) \lb{A.13} \\
&= -k \lb{A.14}
\end{align}
by \eqref{A.3}.
\end{proof}

\begin{proof}[Proof of Proposition~\ref{P2.3}] By Proposition~\ref{PA.1} and \eqref{A.2}, both sides of \eqref{2.4} are
$\idx(P_{(-\infty, E)}(A), P_{(-\infty, E)} (A+B_+ - B_-))$.
\end{proof}

\bigskip

\end{document}